\documentclass[a4paper,10pt]{article}
\usepackage{authblk}

\usepackage[english]{babel}
\usepackage[latin1]{inputenc}
\usepackage{csquotes}
\usepackage[normalem]{ulem}
\usepackage{amsfonts,amsmath,amsthm}
\usepackage{empheq}
\usepackage{cancel}
\usepackage[titletoc,title]{appendix}
\usepackage[titletoc,title]{appendix}
\usepackage[backend=bibtex8,doi=false,eprint=false,giveninits=true,isbn=false,style=numeric-comp,url=false,maxnames=99]{biblatex}
\makeatletter
\def\blx@maxline{77}
\makeatother
\bibliography{bibl.bib}
\DeclareRedundantLanguages{english,german,french}{english,german,ngerman,french}

\usepackage{cases}
\usepackage{mathabx}
\usepackage{xfrac}
\usepackage{fancyhdr}
\usepackage{color}
\usepackage[colorinlistoftodos,textsize=small,backgroundcolor=white,bordercolor=blue,linecolor=blue]{todonotes}
\usepackage[colorlinks]{hyperref}
\definecolor{blue75}{rgb}{0,0,.75}
\definecolor{green75}{rgb}{0,.75,0}
\hypersetup{colorlinks=true, urlcolor=blue75,linkcolor=blue75,citecolor=green75,pdfstartview=FitB,bookmarksopen=true,bookmarksopenlevel=1}

\usepackage[capitalise]{cleveref}

\crefdefaultlabelformat{{\it #2#1#3}}

\crefname{equation}{}{}

\crefname{enumi}{}{}
\crefformat{enumi}{#2#1#3}

\crefname{section}{{\it Section}}{{\it Sections}}
\crefname{subsection}{{\it Section}}{{\it Sections}}
\crefname{subsubsection}{{\it Paragraph}}{{\it Paragraphs}}
\crefname{table}{{\it Table}}{{\it Tables}}

\newtheorem{Theorem}{Theorem}[section]
\crefname{Theorem}{{\it Theorem}}{{\it Theorems}}
\newtheorem{Definition}[Theorem]{Definition}
\crefname{Definition}{{\it Definition}}{{\it Definitions}}
\newtheorem{Lemma}[Theorem]{Lemma}
\crefname{Lemma}{{\it Lemma}}{{\it Lemmas}}

\crefname{Proposition}{{\it Proposition}}{{\it Propositions}}
\newtheorem{Assumptions}[Theorem]{Assumptions}
\crefname{Assumptions}{{\it Assumptions}}{{\it Assumptions}}

\theoremstyle{definition}
\newtheorem{Remark}[Theorem]{Remark}
\crefname{Remark}{{\it Remark}}{{\it Remarks}}

\crefname{Notation}{{\it Notation}}{{\it Notations}}
\newtheorem{Example}[Theorem]{Example}
\crefname{Example}{{\it Example}}{{\it Examples}}

\usepackage[a4paper, left=2.5cm, right=2.5cm, top=2.5cm,bottom=2cm]{geometry}
\usepackage{constants}
\newcommand{\parenthezises}[1]{\arabic{#1}}
\newconstantfamily{C}{
symbol=C,
format=\parenthezises,
}
\newconstantfamily{e}{
symbol=\varepsilon,
format=\parenthezises,
}
\usepackage{enumerate}

\usepackage{graphicx}
\graphicspath{ {images/} }
\usepackage{wrapfig}
\usepackage{figbib}
\allowdisplaybreaks
\begin{document}

\newcommand{\cred}[1]{\textcolor{blue}{#1}}
\newcommand{\cb}[1]{\textcolor{blue}{#1}}

\newcommand{\one}{\mathbf{1}}

\newcommand{\vd}{\varphi_{\delta}}

\newcommand{\io}{\int_{\Omega}}
\newcommand{\id}{\int_{\R^d}}
\newcommand{\R}{\mathbb{R}}
\newcommand{\N}{\mathbb{N}}
\newcommand{\Z}{\mathbb{Z}}
\newcommand{\oa}{\overline{\Omega}}
\newcommand{\ve}{\varepsilon}
\newcommand{\D}{\mathbb{D}}
\newcommand{\dij}{d_{ij}}
\newcommand{\De}{\D_{\varepsilon}}
\newcommand{\Dek}{\D_{\varepsilon_k}}
\newcommand{\As}{\mathcal{A}}
\newcommand{\td}{\,\text{d}}
\newcommand{\dije}{(\dij)_{\varepsilon}}
\newcommand{\dijk}{(\dij)_{\varepsilon_k}}
\newcommand{\ce}{c_{\varepsilon}}
\newcommand{\cen}{c_{\varepsilon_k}}
\newcommand{\coe}{c_{0\varepsilon}}
\newcommand{\Tmaxe}{T_{max,\varepsilon}}
\newcommand{\bc}{\overline{c}}
\newcommand{\cea}{C^{1+\alpha,\frac{1+\alpha}{2}}}
\newcommand{\cza}{C^{2+\alpha,1+\frac{\alpha}{2}}}
\newcommand{\czao}{C^{2+\alpha}(\oa)}
\newcommand{\cek}{c_{\varepsilon_k}}
\newcommand{\ad}{\alpha_{\delta}}
\newcommand{\Dd}{\{\D \geq \delta\}}
\newcommand{\Ddk}{\{\D < \delta_n\}}
\newcommand{\Nd}{N_{\delta}}
\newcommand{\Bad}{B_{\frac{\ad}{2}}}
\newcommand{\en}{\beta_n}
\newcommand{\Ad}{A_{\delta}}

\newcommand{\diag}{\text{diag}}
\newcommand{\dist}{\text{dist}}
\newcommand{\supp}{\text{supp}}
\newcommand{\sign}{\text{sign}}

\newcommand{\Lip}{\text{Lip}}

\numberwithin{equation}{section}

\theoremstyle{definition}
\newtheorem{proofpart}{Step}
\makeatletter
\@addtoreset{proofpart}{Theorem}
\makeatother
\numberwithin{equation}{section}

\title{Global existence of solutions to a nonlocal equation with degenerate anisotropic diffusion}
\author{Maria Eckardt\thanks{Felix-Klein-Zentrum f\"ur Mathematik, RPTU Kaiserslautern-Landau,
Paul-Ehrlich-Str. 31, 67663 Kaiserslautern, Germany \href{mailto:eckardt@mathematik.uni-kl.de}{eckardt@mathematik.uni-kl.de}} \ and Anna Zhigun\thanks{School of Mathematics and Physics, Queen's University Belfast, University Road, Belfast BT7 1NN, Northern Ireland, UK, \href{mailto:A.Zhigun@qub.ac.uk}{A.Zhigun@qub.ac.uk}} 
}

\date{}
\maketitle

\begin{abstract}
 Global existence of very weak solutions to a non-local diffusion-advection-reaction equation is established under no-flux boundary conditions in higher dimensions. The equation features  degenerate myopic diffusion and nonlocal adhesion and is an extension of a mass-conserving model recently derived in  \cite{ZRModelling}. The admissible degeneracy of the diffusion tensor is characterised in terms of the upper box fractal dimension.  
\end{abstract}

\section{Introduction}
 In this paper we study the initial boundary value problem (IBVP)
\begin{subequations}\label{model}
\begin{alignat}{3}
    &\partial_t c = \nabla \nabla : (\D c) - \nabla \cdot (c \As c ) + \mu c (1-c^{r-1}) &&\qquad\text{in }{\Omega\times [0,\infty)},\label{IPDE}\\
    &(\nabla \cdot (\D c)- c\As c)\cdot \nu = 0 &&\qquad\text{in }{\partial\Omega\times [0,\infty)},\label{bc}\\
&c = c_0 &&\qquad \text{in }{\Omega\times \{0\}},
\end{alignat}
\end{subequations}
where $\As$ is the standard adhesion operator \cite{Armstrong2006}{, see \cref{DefA} below}, $\nabla \nabla : ({\D}c)$ is the myopic diffusion \cite{HillenPainter} driven by a symmetric non-negative definite diffusion tensor $\D=\D(x)$, $\mu>0$ and $r\geq2$ are positive constants, and $\Omega$ is a smooth bounded domain in $\R^d$, $d\in\N$,   boundary $\partial\Omega$ and outer unit normal $\nu$. 
The nonlocal diffusion-advection-reaction equation \cref{IPDE} is an extension of an equation that was recently derived in \cite{ZRModelling} using a multiscale approach and   corresponds to the case $\mu=0$. It can describe the evolution of density $c=c(t,x)$ of a cell population that disperses due to a potentially anisotropic diffusion and nonlocal adhesion, thus upgrading the original model from \cite{Armstrong2006} where the diffusion term is $D\Delta c$ with $D>0$ a constant. We refer to \cite{ZRModelling} for further details regarding the modelling and derivation approaches.

While the combination of adhesion with a Fickian-type diffusion has received much attention, see, e.g.  \cite{reviewNonlocal2020} and references therein, the case of myopic diffusion has not been analysed so far. The few papers \cite{Heihoff,WinSur2017,WinklerMyopic,StinnerWinklerMyopic} that have dealt with existence and long-time behaviour of solutions to problems that include both myopic diffusion and advection are restricted to versions of the model derived in \cite{EngwerHuntSur}. It features haptotaxis, i.e. the directed movement along the local  gradient of an external immovable signal, rather than the spatially nonlocal  intrapopulational adhesion as in \cref{IPDE}. 
  Apart from that, as a result of somewhat different underlying derivation approaches in \cite{EngwerHuntSur} compared to \cite{ZRModelling}, the advection velocity in the aforementioned haptotaxis model is multiplied by the diffusion tensor, whereas in \cref{IPDE} this is not the case. Thus, here it is in no way 'subordinate' to the diffusion and, in particular, the adhesion term need not vanish in those areas where diffusion is absent. Finally, we observe that apart from \cite{Heihoff} where dimensions two and three were treated, other  works \cite{WinSur2017,WinklerMyopic,StinnerWinklerMyopic} only considered the one-dimensional case.

The goal of this note is to establish a result on global existence of solutions to \cref{IPDE} equipped with no-flux boundary and initial conditions. Our approach works for $\mu>0$, i.e. in the presence of the generalised logistic-type growth term. While it describes  a biologically relevant effect (e.g. cell growth/death), our main motivation for including the source term stems from the analytical challenges that arise in the case of $\mu=0$. In the latter scenario, since the diffusion is non-Fickian and degenerate, only mass preservation is a priori guaranteed, indicating that generally solutions need not be functions but could be measure-valued. Here we chose to avoid this possibility by including the growth term. While our analysis allows for degenerate diffusion tensors, we require the degeneracy set, i.e. the set of points where $\D$ is not positive definite, to have a positive distance to  the boundary of $\Omega$ and to be sufficiently low-dimensional, see condition \cref{bedd} below. This condition seems to be new in the context of degenerate diffusion. It  arises from \cref{LemFD} in \cref{appB} and provides a certain balance between the degenerate diffusion and the nonlinear growth term.   

The reminder of the paper is organised in the following way. After fixing some notation and formally defining the adhesion operator $\As$  in \cref{SecNot}, we fully set-up our model and formulate our main result on existence of very weak solutions,  \cref{veryweaksol}, in \cref{SecModel}. We then analyse suitably constructed approximation problems in \cref{SecAppr}. The uniform estimates that we establish there allow to apply the compactness method and prove \cref{veryweaksol} in \cref{SecEx}. Finally, in \cref{appendixvws} we provide a justification of our solution concept proving that regular solutions of this sort are classical. 

\section{Preliminaries}\label{SecNot}
\subsection{Notation}
 {Let $d\in\N$, $d>1$.} For $x\in\R^d$ we denote $x=(x',x_d)$, where $x' \in \R^{d-1}$ and $x_d \in \R$. 
 
 {By $B_1$ we denote the unit ball in $\R^d$ centred at the origin.}
 
We use the set notations   
\begin{align}
&A+B:=\{x+y:\ x\in A,\ y\in B\},\\
&a+B:=B+a:=\{a\}+B,\\
 &O_{s}(A):=\{x\in\R^d:\ \dist(x;A)<s\}
\end{align}
for $A,B\subset\R^d$, $a\in\R^d$, and $s>0$.

In $\R^d$, we denote by $|\cdot|$ and $|\cdot|_{\infty}$ the Euclidean and infinity norms, respectively.
{For matrices $A,B \in \R^{d\times d}$ we use the Frobenius inner product
\begin{align*}
	A : B = \sum_{i,j=1}^{d} a_{ij}b_{ij}.
\end{align*}
We denote by $|A|$ the spectral norm of $A$.
}

 {Let $\Omega\subset \R^d$ be a bounded domain with a sufficiently smooth  boundary.} By convention{,} we assume for any function $u: \Omega \rightarrow \R$ that $u \equiv 0$ on $\R^d\setminus \oa$. This allows for an obvious meaning to be given to  convolution $u\star v$ for any $u\in L^1(\Omega)$ and $v\in L^1(\R^d)$.  
These conventions extend   componentwise to any vector{/}matrix-valued function $u$.

{For} a matrix function $\D = (\dij)_{i,j = 1,{\dots},d} \in C(\oa;\R^{d\times d})$ we {write} $\D \geq c$ if $x^T\D(y) x \geq c$ for all $x \in \R^d$ and $y\in \oa$ (analogously for $\geq$, $<$, $\leq$). Further, for $\D \geq 0$ we define the set 
\begin{align*}
	\{\D \ngtr 0\} := \{x \in \oa : \, \exists y\in\R^n \text{ s.t. } y^T\D(x) y = 0\}.
\end{align*}

For the definition of high-order H\"older spaces we refer to \cite{Lady}.

 Finally, we make the convention that $C_i$ denotes a positive constant or a positive function of its arguments for all indices $i\in \N$.

\subsection{The adhesion operator}
We formally define the adhesion $\As$  operator  between two functional spaces in the way that suits our needs.
\begin{Definition}\label{DefA}
Consider a continuous function $F : [0,1] \rightarrow [0,\infty)$.
The adhesion operator is given by
\begin{align}
	\As : L^1(\Omega) \rightarrow \left(L^{\infty}(\Omega)\right)^{ d}, \quad \As u (x) :=  \frac{1}{|B_1|}\int_{B_1} u(x+ \xi) \frac{\xi}{|\xi|} F(|\xi|) d\xi.\label{DefAs}
\end{align}
\end{Definition}
{\begin{Remark}In \cref{DefAs} we assumed the so-called sensing radius, i.e. the radius of the integration domain, to be $1$. This is no restriction of generality since this value can be always achieved through a suitable rescaling of the spatial variable.\end{Remark}}

 The operator is  well-defined and bounded, see, e.g. \cite{EckardtPaSuZh}. Further, as observed in \cite{ZRModelling}, $$\As u = -\nabla H \ast u,$$ where $H$ is the interaction potential given by
\begin{align*}
	H: \R \rightarrow [0,\infty], \quad H(x) := \frac{1}{|B_1|} \int_{\min\{|x|,1\}}^{1} F(\sigma) \,\text{d}\sigma,
\end{align*}
so that its gradient is the $L^{\infty}$ function
\begin{align*}
	\nabla H (x) := \begin{cases}
		-\frac{1}{|B_1|} \frac{x}{|x|}F(|x|)  &\text{ if }x \in B_1\setminus \{0\},\\
		0 &\text{ if }x \in \R^d\setminus \overline{B_1}.
	\end{cases}
\end{align*}
Consequently, we can state the Lemma below which follows directly from the properties of convolution and the definition of a H\"older space (see e.g. \cite[Chapter 6, Theorems 6.27-6.28]{Klenke} for $C^k$ with $k\in\N$, the H\"older estimates follow through a direct calculation). 

\begin{Lemma}\label{lemconvstetig}
Let $k \in \N$ and $\alpha \in [0,1]$. Then, $\As$ is a continuous operator in 
$C^{k,\alpha}(\oa)$.
\end{Lemma}

\section{Problem setting and main result}\label{SecModel}

We make the following assumptions on the coefficients and other parameters.
\begin{Assumptions} \label{assump}
\begin{align}
 &d \geq 3,\ r > \frac{d}{d-2},\ r \geq 2,\label{Assdr}\\
 &\Omega \subset \R^d\text{ is a domain with }C^3\text{ boundary and outer unit normal }\nu,\nonumber\\
 &F\in C([0,1];[0,\infty)),\nonumber\\
&\mu > 0,\nonumber\\
&c_0 \in L^r(\Omega),\nonumber
\end{align}
and 
\begin{subequations}
\begin{align}
&\D := (\dij)_{i,j = 1,\dots,d}  \text{ symmetric},\\
	&\D \in C(\oa; \R^{d\times d}),\\
	&\nabla \cdot \D \in \left(L_{\text{loc}}^{\infty}(\{\D > 0 \})\right)^{ d},\\
	&\D \geq 0,\\
	&a := \dist\left(\partial \Omega, \{\D \ngtr 0\} \right) > 0, \label{distdod}\\
	&\overline{\dim}_{\cal F}(\{\D \ngtr 0\})<d-\frac{2r}{r-1}. \label{bedd}
\end{align}
\end{subequations}	
\end{Assumptions}
 To illustrate condition \cref{bedd}, we consider two special cases where it is satisfied. The upper box dimension $\overline{\dim}_{\cal F}(\dots)$ and sets $Z_{\delta}(\dots)$ (see below) are introduced in \cref{DefUB}, see \cref{appB}.
\begin{Example}
\begin{enumerate}
	\item For a finite set $\{\D \ngtr 0\} = \{a_1,\dots,a_N\}$ with $N \in \N$ we can estimate
	\begin{align*}
		|Z_{\delta}(\{\D \ngtr 0\})| = \left| \left\{b \in \delta \Z^d: |a_i-b|_{\infty} \leq \frac{\delta}{2} \text{ for some } i = 1, \dots,N\right\}\right| \leq 2N.
\end{align*}
Then,
\begin{align*}
\overline{\dim}_{\cal F}(\{\D \ngtr 0\}) = \underset{\delta\to0}{\lim\sup}\frac{\log_2|Z_{\delta}(K)|}{\log_2\delta^{-1}} \leq \underset{\delta\to0}{\lim\sup} \frac{\log_2(2N)}{\log_2(\delta^{-1})} = 0.
\end{align*}
	\item Consider a sequence $\{a_n\}_{n\in\N}$ in $\Omega$ with $\lim_{n\to\infty} a_n = a \in \Omega$, i.e. for all $\delta >0 $ there is $N(\delta)\in \N$ s.t. $|a_n-a|< \frac{\delta}{2}$ for all $n \geq N(\delta)$. We assume that the sequence converges fast enough to its limit in the sense that there is $b < d-\frac{2r}{r-1}$ s.t.
	\begin{align}\label{limfol}
		\underset{\delta\to0}{\lim\sup} \frac{\log_2(N(\delta))}{\log_2(\delta^{-1})} = b.
	\end{align}
	If $\{\D \ngtr 0\} = \{a,a_0,a_1,\dots\}$, then
	\begin{align}
		|Z_{\delta}(\{\D \ngtr 0\})| \leq 2(N(\delta) + 1).\label{limfol2}
	\end{align}
Combining \cref{limfol,limfol2}, we arrive at \cref{bedd}.
\end{enumerate}	
\end{Example}

 We define solutions to the IBVP \cref{model} as follows.
\begin{Definition} \label{defvws}
We call a function $c \in L^r_{loc}(\oa \times [0,\infty))$ a \underline{global very weak solution} to \cref{model} if it satisfies
	\begin{align}
		&- \int_0^{\infty} \io c  \partial_t \eta dxdt - \io c_0 \eta(\cdot,0)dx\nonumber\\ 
		=&  \int_0^{\infty} \io c\D: D^2\eta  \td{x} \td{t} + \int_0^{\infty} \io c(\As c) \cdot \nabla \eta \td{x} \td{t} + \mu \int_0^{\infty} \io c (1-c^{r-1}) \eta \td{x} \td{t}\label{WeakF}
	\end{align}
for all
	\begin{align*}
	 \eta \in C_{c}^{2,1}(\oa \times [0,\infty))\qquad\text{s.t. }\nabla \eta \cdot (\D \nu) \equiv 0\text{
on }\partial \Omega \times (0,\infty).
\end{align*}
\end{Definition}

 Our main result  concerns with the existence of such solutions.
\begin{Theorem} \label{veryweaksol}
Let Assumptions \ref{assump} hold. Then, there is a very weak solution 
$$c\in L^{\infty}(0,\infty;L^1(\Omega))\cap L^{r}_{loc}(\oa\times[0,\infty))$$ to \cref{model} in the sense of \cref{defvws}.
\end{Theorem}

\begin{Remark} The very weak formulation \cref{WeakF} is obtained by multiple partial integration that shifts all spatial derivatives to the test function. This choice of formulation exploits the structure of the myopic diffusion. At the same time, it avoids including terms such as $\nabla\cdot(\D c)$ or $\nabla c$, for which it is likely not possible  to obtain a priori bounds in the whole domain $\Omega$ due to a combination of the degeneracy of the diffusion tensor $\D$ and the diffusion being myopic.

 Furthermore, we show in \cref{appendixvws} that sufficiently smooth very weak solutions are classical solutions to \cref{model}. This justifies our solution concept.
\end{Remark}

\section{Approximate problems}\label{SecAppr}
\subsection{Construction of  a regular matrix family \texorpdfstring{$(\De)$}{}}\label{SecDe}
We begin by constructing a family of regular matrices $(\De)$ that approximate $\D$ in a suitable fashion. For such diffusion matrices, existence of regular solutions can be directly concluded from known results. Unlike \cite{Heihoff}, we impose uniform $L^{\infty}$ boundedness of the divergence for $(\De)$  instead of convergence in an $L^p$ space for finite $p$ and make no additional restrictions  such as, e.g. vanishing normal trace, because our analysis does not require this.
\begin{Lemma} \label{lemapproxd}
Let $\Omega$ be Lipschitz.
For $\Cl[e]{e1}>0$ small enough, there is a sequence of symmetric matrices $(\De)_{\ve \in (0,\Cr{e1})} \subset C^{\infty}(\oa; \R^{d\times d})$ s.t.
\begin{alignat}{3} 
&\|\De\|_{(L^{\infty}(\Omega))^{d\times d}}\leq\ve+\|\D\|_{(L^{\infty}(\Omega))^{d\times d}}&&\qquad\text{for }\ve\in(0,\Cr{e1}),\label{estDe}\\
	&\De\geq \ve&&\qquad\text{for }\ve\in(0,\Cr{e1}),\\
	&\|\De - \D\|_{(C(\oa))^{d\times d}} \rightarrow 0 \text{ as } \ve \rightarrow 0.&&\label{detod}
\end{alignat}
Further, for every relatively open 
\begin{align*}
 B \subset\subset { \{\D>0\}}
\end{align*}
there exists some $\Cr{e2}(B)\in(0,\Cr{e1}]$ such that 
\begin{align}\label{bounddivepsb}
	\|\nabla \cdot \De\|_{(L^{\infty}(B))^d} \leq \Cl{divde}(B)\qquad\text{ for }\ve\in(0,\Cr{e2}(B)).
\end{align}
\end{Lemma}
\begin{proof} We take a standard approach, arguing in a manner similar to the proof of Theorem 3 in Section 5.3.3 in \cite{Evans}.
\begin{proofpart} 
We need some preparation before we can proceed with a regularisation of $\D$. It concerns the domain $\Omega$. Let $x_0 \in \partial \Omega$. Since $\Omega$ has a Lipschitz boundary, there exist $\gamma \in C^{ 0,1}(\R^{d-1}{;} \R)$ and $\rho\in (0,a)$ s.t. (after relabeling and reorientation of the axes if necessary)
\begin{align}
	\Omega \cap B_{\rho}(x_0) &= \left\{ y \in B_{\rho}(x_0): y_d > \gamma (y')\right\}, \label{defmenge}\\
	\partial \Omega \cap B_{\rho}(x_0) &= \left\{ y \in B_{\rho}(x_0): y_d = \gamma (y')\right\}. \label{defrand}
\end{align}
Due to $\rho< a$, the set $\Omega \cap B_{\rho}(x_0)$ does not intersect $\{\D\ngtr 0\}$. 
Let us check that 
\begin{align}
 &\Omega \cap B_{\frac{\rho}{2}}(x_0)+B_{\ve}(\ve(L+1) e_d)\subset \Omega \cap B_{\rho}(x_0)\label{condball}
 \end{align}
 for
 \begin{align}
 \ve\leq\frac{\rho}{2(L+2)}{,}
 \label{condep1}
\end{align}
where $L$ is a Lipschitz constant for $\gamma$ in 
\begin{align*}
	B_{\rho}^{d-1}(x_0') := \left\{z' \in \R^{d-1}: \ |z'-x_0'|< \rho\right\}.
\end{align*}
Let $y\in \Omega \cap B_{\frac{\rho}{2}}(x_0)$ and $z\in B_{\ve}(\ve(L+1) e_d)$. Due to \cref{condep1}, it holds that
\begin{align}
 |y+z-x_0|\leq &|y-x_0|+|z-\ve(L+1) e_d|+\ve(L+1) \nonumber\\
 <&\frac{\rho}{2}+(L+2)\ve\nonumber\\
 \leq&\rho.\label{inball}
\end{align}
 As $z\in B_{\ve}(\ve(L+1) e_d)$ implies that $|z'|<\varepsilon$ and $z_d > \varepsilon(L+1)-\ve$, the Lipschitz continuity of $\gamma$ and \cref{defmenge} together imply
\begin{align}
 \gamma(y'+z')\leq &\gamma(y')+L|z'|\nonumber\\
 <&y_d+\ve L\nonumber\\
 <&y_d+z_d-\ve((L+1)-1)+\ve L\nonumber\\
 =&y_d+z_d.\label{inOm}
\end{align}
Combining \cref{inball,inOm,defmenge}, we arrive at \cref{condball}.

 Due to compactness of $\overline{\Omega}$, we can find some constants $\rho\in(0,a)$, $L>0$, and $k\in\N$ and  points $x_k \in \partial{\Omega}$ and $z_k\in S_1(0)$, $k\in\{1,\dots,K\}$, such that
\begin{align}
	&{\Omega\cap}B_{\frac{\rho}{2}}(x_k)+B_{\ve}(\ve(L+1) z_k)\subset B_{\rho}(x_k)\cap\Omega\qquad \text{for }k\in\{1,\dots,K\},\ \ve\in\left(0,{\frac{\rho}{2(L+2)}}\right),\label{condballk}
\end{align}
and
\begin{align*}
	&\partial\Omega\subset \bigcup_{k=1}^K B_{\frac{\rho}{2}(x_k)}.
\end{align*}
Let 
\begin{align*}
 A_0:=\oa\backslash \left(\bigcup_{n=1}^K B_{\frac{\rho}{2}(x_k)}\right).
\end{align*}
This set is compact  and satisfies {$$\rho_0 := \text{dist}(A_0,\partial \Omega)>0.$$}
By Theorem 3.15 in \cite{Adams}, there is a partition of unity $\{\psi_k\}_{k=0}^K$ subordinate to the open covering 
\begin{align}
O_{\frac{\rho_0}{2}}(A_0),\   \{B_{\frac{\rho}{2}}(x_k)\}_{k=1}^{K}\label{cov}                                                                   \end{align}
of $\overline\Omega$, i.e. a set of functions that satisfies
\begin{subequations}\label{psis}
\begin{alignat}{3}
&\psi_0\in C_{ c}^{\infty}(O_{\frac{\rho_0}{2}}(A_0)),&&\\
&\psi_k \in C_{ c}^{\infty}(B_{\frac{\rho}{2}}(x_k))&&\qquad\text{for }k\in\{1,\dots,K\},\\
&0\leq \psi_k \leq 1 &&\qquad\text{for }k\in\{0,\dots,K\},\\
&\sum_{k=0}^K \psi_k = 1 \qquad\text{in }\oa.&&                                                                                                                                                                                                                                                                                                                                                             \end{alignat}
\end{subequations}
\end{proofpart}
\begin{proofpart}
 Now we can proceed with the construction of approximations for $\D$. Set
\begin{align}
  \De(x):=\ve I_d+\psi_0{ (x)}({ \eta_{\ve} \ast \D})(x)+\sum_{k=1}^K\psi_k{ (x)}({ \eta_{\ve} \ast \D}) (x + \ve(L+1)z_k)\qquad\text{for }x\in\oa,\ \ve\in\left(0,\Cr{e1}\right),\label{DefDe}
\end{align} 
{where $$\Cr{e1} := \min\left\{\frac{\rho}{2(L+2)},{\frac{\rho_0}{3}}\right\},$$}$\eta$ denotes the standard mollifier{,} and $\eta_\ve(y) := \frac{1}{\ve^d}\eta(\frac{y}{\ve})$.
Obviously, $\De \in C^{\infty}(\oa, \R^{d\times d})$, is symmetric in $\oa$, and $\De \geq \ve$. Using Young's inequality for convolutions and properties of $\eta_{\ve}$  and $\psi_k$s, one readily obtains \cref{estDe}. \\ 
Further, thanks to \cref{condballk} we can exploit uniform continuity of $\D$ in $\oa$, yielding 
\begin{align}
  &\max_{x\in \overline{B_{\frac{\rho}{2}(x_k)}\cap \Omega}}\left|({ \eta_{\ve} \ast \D})(x + \ve(L+1)z_k)-\D(x)\right|\\
  =&\max_{x\in \overline{B_{\frac{\rho}{2}(x_k)}\cap \Omega}}\left|\int_{B_{\ve}(0)}\eta_{\ve}(y)(\D(x+\ve(L+1)z_k-y)-\D(x))\,dy\right|\nonumber\\
  \leq&\max_{(x,y)\in \overline{B_{\frac{\rho}{2}(x_k)}\cap \Omega}\times  \overline{B_{\ve}(0)}}|(\D(x+\ve(L+1)z_k-y)-\D(x))|
  \underset{\ve\to0}{\to}0\qquad\text{for }k\in\{1,\dots,K\}.\label{conv1}
\end{align}
Since $\overline{O_{ \frac{\rho_0}{2}}(A_0)}\subset\Omega$, we also have (e.g. due to Theorem 2.29(d) in \cite{Adams})
\begin{align}
  { \eta_{\ve} \ast \D}\underset{\ve\to0}{\to}\D\qquad\text{in }\left(C(\overline{O_{ \frac{\rho_0}{2}}(A_0)})\right)^{d\times d}.\label{conv0}
\end{align}
Combining \cref{conv0,conv1,psis}, we arrive at
\begin{align*}
  \De\underset{\ve\to0}{\to}\D\qquad\text{in }(C(\oa))^{d\times d}.
\end{align*}
\end{proofpart}
\begin{proofpart} 
It remains to verify \cref{bounddivepsb}.
Consider some relatively  open set 
$B \subset\subset { \{\D>0\}}$ and assume 
\begin{align}
 &B \subset B_{\frac{\rho}{2}}(x_k)\label{AssB}
\end{align}
for some $k\in\{1,\dots,K\}$. Then,
\begin{subequations}\label{supp1}
\begin{align}
 &\overline{B}+\overline{B_{\ve}(0)}+\ve(L+1)z_k\subset \overline{O_{\ve(L+2)} (B)}{\cap \oa} \subset\subset \{\D>0\}\cap B_{\rho}(x_k)\qquad\text{for }\ve\in(0,\Cr{e2}(B)),\\
 &\Cl[e]{e2}(B):=\min\left\{\frac{\rho}{2(L+2)},\frac{1}{{2}(L+2)}\dist(B,\{\D \ngtr 0\}),{\frac{\rho_0}{3}}\right\}.
\end{align}
\end{subequations}
Let $\varphi\in C^{\infty}_{ c}(B{ \setminus \partial \Omega})$. By \cref{supp1}, the definition of weak divergence and properties of convolutions and $\eta_{\ve}$ we have
\begin{align}
	\id \nabla \cdot ({ \eta_{\ve} \ast \D})(x + \ve(L+1)z_k) \varphi \td x =&
  -\id({ \eta_{\ve} \ast \D})(x + \ve(L+1)z_k)\nabla\varphi(x)\td{x}\\
  =&-\id \D(x) (\eta_{\ve}\ast\nabla\varphi)(x-\ve(L+1)z_k)\td{x}\nonumber\\
  =&-\id \D(x) \nabla((\eta_{\ve}\ast\varphi)(x-\ve(L+1)z_k))\td{x}\nonumber\\
  =&\id \nabla\cdot \D(x) (\eta_{\ve}\ast\varphi)(x-\ve(L+1)z_k)\td{x}\nonumber\\
  =&\int_{O_{\ve(L+2)}{ (B) \cap \Omega}}\nabla\cdot \D(x) (\eta_{\ve}\ast\varphi)(x-\ve(L+1)z_k)\td{x},\nonumber
\end{align}
so that
\begin{align}
 \left|\id \nabla \cdot ({ \eta_{\ve} \ast \D})(x + \ve(L+1)z_k) \varphi \td x\right|_{\infty}\leq&\left\|\nabla\cdot\D\right\|_{(L^{\infty}(O_{\ve(L+2)}{ (B) \cap \Omega}))^d}\|\varphi\|_{L^1(B)}\nonumber\\\leq&{\left\|\nabla\cdot\D\right\|_{(L^{\infty}(O_{\ve_2(B)(L+2)}(B) \cap \Omega))^d}\|\varphi\|_{L^1(B)}}.\nonumber
\end{align}
Consequently, due to the density of $C_{ c}^{\infty}$ in $L^1$ and as $(L^1)^* = L^{\infty}$, the estimate
\begin{align}
 \left\|\nabla\cdot({ \eta_{\ve} \ast \D})(\cdot+\ve(L+1)z_k))\right\|_{(L^{\infty}(B))^d}\leq {\sqrt{d}}\left\|\nabla\cdot\D\right\|_{(L^{\infty}(O_{{ \ve_2(B)}(L+2)}{ (B) \cap \Omega}))^d}\qquad\text{for }\ve\in(0, \Cr{e2}(B))\label{estdiv1}
\end{align}
follows. A similar argument works for ${ \eta_{\ve} \ast \D}$, yielding 
\begin{align}
 \left\|\nabla{\cdot}({ \eta_{\ve} \ast \D})\right\|_{(L^{\infty}(B))^d}\leq {\sqrt{d}}  \left\|\nabla\cdot\D\right\|_{(L^{\infty}({ O_{\ve_2(B)}(B)}))^d}\qquad\text{for }\ve\in(0,\Cr{e2}(B)).\label{estdiv2}
\end{align}
Combining \cref{DefDe,estdiv1,estdiv2,estDe} and using the product rule, we arrive at  \cref{bounddivepsb}. The case $B \subset \subset \{\D > 0\}$ { and $B \subset O_{\frac{\rho_0}{2}}(A_0)$} can be handled analogously. In general, any $B \subset\subset {\{\D>0\}}$ {can be described as the union of at most $K+1$  sets such that each of them is} fully contained in (at least) one of the elements of the open covering \cref{cov}. {Therefore,} \cref{bounddivepsb} still holds if we choose $\Cr{e2}(B)$ and $\Cr{divde}(B)$ to be sufficiently small and large, respectively.
\end{proofpart}
\end{proof}

\subsection{Existence of a global classical solution to the approximate problem}
Let $\alpha \in (0,\frac{1}{2})$. Since $C_c^{\infty}(\Omega)$ is dense in  $L^2(\Omega)$, there is a sequence $(\coe)_{\ve\in(0,\Cr{e1})} \subset C^{2+\alpha}(\oa)$ and a constant $\Cl{coeb}$ satisfying 
\begin{alignat}{3}
&(\nabla \cdot (\De \coe) - \coe\As \coe)\cdot \nu = 0&&\qquad\text{in }\partial\Omega,\nonumber\\
&\coe\underset{\ve\to0}{\to}c_0&&\qquad\text{in }L^2(\Omega),\label{convc0}\\
&\|\coe\|_{L^2(\Omega)}\leq \Cr{coeb}&&\qquad\text{for all }\ve>0.\nonumber
\end{alignat}
With the help of the diffusion tensors constructed in \cref{SecDe} we formulate for $\ve \in (0,\Cr{e1})$ the approximate problems 
\begin{subequations}\label{modelapprox}
\begin{alignat}{3}
    &\partial_t \ce = \nabla \nabla : (\De\ce) - \nabla \cdot (\ce \As \ce ) + \mu \ce (1-\ce^{r-1}) &&\qquad\text{in }{\Omega\times [0,\infty)},\label{IPDEe}\\
    &(\nabla \cdot (\De \ce)- \ce\As \ce)\cdot \nu = 0 &&\qquad\text{in }{\partial\Omega\times [0,\infty)},\label{bce}\\
&\ce = \coe &&\qquad \text{in }{\Omega\times \{0\}}.
\end{alignat}
\end{subequations}
The subsequent \cref{LemLoc,LemGl} provide local and then global existence of classical solutions to \cref{modelapprox}.
\begin{Lemma}\label{LemLoc}
{Let $\alpha \in (0,\frac{1}{2})$.} For $\ve \in (0,\Cr{e1})$ there is a maximal existence time $\Tmaxe \in (0,\infty]$ and a nonnegative classical solution $\ce \in C^{2{+\alpha},1{+\frac{\alpha}{2}}}(\oa \times [0,\Tmaxe))$ of \cref{modelapprox}. It holds that either $\Tmaxe = \infty$ or $\Tmaxe < \infty$ and 
\begin{align}\label{limceta}
	\lim_{t \nearrow \Tmaxe}\|\ce(\cdot,t)\|_{C^{2+\alpha}(\oa)}=\infty.
\end{align}
\end{Lemma}

\begin{proof}
The proof is based on a standard fixed-point argument.

 Let $\ve \in (0,{\Cr{e1}})$ and $\alpha\in(0,\frac{1}{2})$, as above, $T{ \in (0,1)}$ small enough (to be determined later), and 
 $$M:=\|\coe\|_{\czao} + 1.$$
 We define the set
\begin{align*}
	S := \left\{\bc \in \cea (\oa \times [0,T];[0,\infty)) :\quad  \|\bc\|_{\cea(\oa \times (0,T))} \leq M\right\}.
\end{align*}
For $\bc \in S$ we consider the linearised IBVP 
\begin{subequations}\label{modelapproxf}
\begin{alignat}{3}
    &\partial_t \ce = \nabla \nabla : (\De\ce) - \nabla \cdot (\ce \As \bc ) + \mu \ce (1-\bc^{r-1}) &&\qquad\text{in }{\Omega\times [0,\infty)},\\
    &(\nabla \cdot (\De \ce)- \ce\As \bc)\cdot \nu = 0 &&\qquad\text{in }{ \partial\Omega \times [0,\infty)},\label{bcef}\\
&\ce = \coe &&\qquad \text{in }{\Omega \times \{0\}}.
\end{alignat}
\end{subequations}
Due to \cref{lemconvstetig,lemapproxd}, the coefficient functions and the initial data are smooth enough and satisfy the compatibility condition{, allowing us to conclude with} Theorem IV.5.3 in \cite{Lady} that \cref{modelapproxf} has a unique solution $\ce \in \cza(\oa\times[0,T])$ {for which}
\begin{align*}
\|\ce\|_{\cza(\oa\times[0,T])} \leq \Cl{boundcza}(M,T)\|\coe\|_{\czao}
\end{align*}
for some constant $\Cr{boundcza}(M,T)>0$ independent of the choice of  $\bc$. From Theorem 13.5 (and the remark at the end of Chapter 13 in the book to include some dependence of the coefficients on $t$) in \cite{KochDa} we conclude that $\ce \geq 0$. A straightforward calculation leads to the estimate
\begin{align}
	\|\ce\|_{\cea(\oa\times[0,T])} 
	\leq& \|\ce - \coe\|_{\cea(\oa\times[0,T])} + \|\coe\|_{{ C^{1+\alpha}(\oa)}}\nonumber \\
	\leq& {2\max\{1,\Cr{embsob}\} }T^{\frac{\alpha}{2}} \|\ce\|_{\cza(\oa\times[0,T])} + \|\coe\|_{\czao},\label{Hest}
\end{align}
{where $\Cl{embsob}$ denotes the Sobolev embedding constant from $W^{1,\infty}(\Omega)$ into $C^{\alpha}(\oa)$.}
Hence, {taking $$T \leq \left (\frac{1}{{ 2\max \{1,C_{4}\}} C_{3}(M,T)\|c_{0\varepsilon }\|_{C^{2+\alpha }(\overline{\Omega})}}
\right )^{\frac{2}{\alpha }}$$ we ensure that} $\ce \in S$. 
Consequently, the operator 
\begin{align}\label{defopf}
F: {S \rightarrow S},\qquad \bc \mapsto \ce                                                                                                 
\end{align}
is a well-defined and{,} due to $\cza(\oa \times [0,T])\subset\subset\cea(\oa\times[0,T])${, compact} self-map. Moreover, the continuous dependence of $\ce$ from the coefficients (that follows with \cref{lemconvstetig} and the proof of Theorem IV.5.3 in \cite{Lady}) implies that $F$ is a continuous and compact operator. 
Now, Schauder's fixed point theorem applies, providing a fixed point $\ce \in \cea(\oa\times[0,T])$ of $F$ that is also in $\cza(\oa\times[0,T])$ and a classical solution to \cref{modelapprox} on $\oa\times [0,T]$.

 Extending the solution to its maximal existence time $\Tmaxe$, it holds that either $\Tmaxe = \infty$ or $\Tmaxe<\infty$ and then \cref{limceta}.
\end{proof}

 Next, we verify global existence for \cref{modelapprox}.
\begin{Lemma}\label{LemGl}
{Let $\alpha \in (0,\frac{1}{2})$.} For $\ve \in (0,\Cr{e1})$ there is a nonnegative classical solution $\ce \in C^{2{+\alpha},1{+\frac{\alpha}{2}}}(\oa \times [0,\infty))$ of \cref{modelapprox}. 
\end{Lemma}

\begin{proof}
We follow a standard approach which is based on excluding the possibility of \cref{limceta}.
Let $\ve \in (0,\Cr{e1})$ and assume $\Tmaxe < \infty$, so that $\ce$ is a solution to \cref{modelapprox} in $\oa\times[0,\Tmaxe)$.
Integrating the first equation in \cref{modelapprox} over $\Omega$, we 
conclude from Gronwall's inequality that
\begin{align} \label{boundce}
	\|\ce\|_{L^{\infty}(0,\Tmaxe;L^1(\Omega))} \leq e^{\mu\Tmaxe}\|\coe\|_{L^1(\Omega)} =: \Cl{emtco}(\Tmaxe).
\end{align}
{A standard argument based on the propagation of $L^p$ bounds and the fact that $\nabla H$ and $\nabla \cdot \De$ are essentially bounded ensures that 
\begin{align*}
 \ce\in L^{\infty}(\Omega\times (0,\Tmaxe)).
\end{align*}
Due to non-degeneracy of the equation ($\De \geq \ve$) this, in turn, readily implies
\begin{align*}
 \ce \in L^2(0,\Tmaxe;H^1(\Omega))\cap C([0,\Tmaxe];L^2(\Omega)).
\end{align*}
}%
Now, Theorem 4 from \cite{DiBenedetto} applies and yields $$\ce \in C^{\gamma,\frac{\gamma}{2}}(\oa \times [0,\Tmaxe])$$ for some $\gamma \in (0,\alpha)$. Next, using the H\"older continuity of the coefficients (which holds especially due to \cref{lemconvstetig}) again, 
{we can avail of Theorem 1.2 from \cite{Lieberman1987} which implies
\begin{align*}
 \ce\in& C^{1+\gamma,\frac{1+\gamma}{2}}(\oa \times [0,\Tmaxe])\\
 \subset& C^{\alpha,\frac{\alpha}{2}}(\oa \times [0,\Tmaxe]).
\end{align*}
Yet another application of this theorem now provides
\begin{align*}
 \ce\in& C^{1+\alpha,\frac{1+\alpha}{2}}(\oa \times [0,\Tmaxe]).
\end{align*}
Finally,} Theorem IV.5.3 from \cite{Lady} yields $$\ce \in C^{2+{ \alpha}, 1+\frac{{ \alpha}}{2}}(\oa \times [0,\Tmaxe]),$$ 
 contradicting \cref{limceta}.
\end{proof}

\section{Existence of a very weak solution to the original problem}\label{SecEx}
In this section, we  show the convergence of a suitably chosen sequence of the classical solutions to \cref{modelapprox} to a very weak solution to \cref{model} in the sense of \cref{defvws}. We start with some basic uniform estimates of $(\ce)$.
\begin{Lemma} \label{lemboundlelr}
For all $\ve \in (0,\Cr{e1})$ it holds that 
\begin{align}
	&\|\ce\|_{L^{\infty}(0,\infty;L^1(\Omega))}\leq\Cl{Cr3},\label{BndCr3}\\
	&\|\ce\|_{L^r(\Omega\times (0,T))} \leq \Cl{Cr4}+T{ \mu}\Cr{Cr3}\qquad\text{for }T>0.\label{BndCr32}
\end{align} 
\end{Lemma}
\begin{proof}
Integrating  \cref{IPDEe} over $\Omega$, by parts where necessary, using the boundary conditions, the assumption $r>1$, and the boundedness of $\Omega$, we obtain
\begin{align}
	\frac{d}{dt} \io\ce\td{x} = &\mu \io \ce-\ce^r \td{x}\label{estcel1}\\
	\leq& \Cl{C1r}-\Cl{C2r} { \mu} \io\ce\td{x}. \nonumber
\end{align}
Consequently, Gronwall's inequality implies \cref{BndCr3}. Integrating \cref{estcel1} over $(0,T)$ and using \cref{BndCr3}  immediately yields \cref{BndCr32}.
\end{proof}

 Next, we establish some uniform estimates for derivatives of $\ce$s. For the spatial gradient, the estimates hold away from the degeneracy set of $\D$.
\begin{Lemma} \label{lembounds}
For any $T>0$ and relatively open $B \subset\subset {{ \{\D>0\}}}$
there is a constant $\Cl{crp1}(B,T)$ s.t.
\begin{align}
	&\|\ce\|_{L^2(0,T;H^1({B}))}\leq \Cr{crp1}(B,T),\label{estLA_}\\
	&\|\ce \|_{L^{r+1}({B}\times(0,T))}\leq\Cr{crp1}(B,T)\label{est3r}
\end{align}
for all $\ve \in (0,{ \Cr{e2}}(B))$. %

 Furthermore, let $q$ be a number that satisfies
\begin{align}
 q\in \left(1,{\frac{d}{d-1}}\right). \label{Assq}
\end{align}
Then, for any $T>0$ there exists some constant $\Cl{Crt}(T)>0$ s.t.
\begin{align}
 \|\partial_t \ce\|_{L^1(0,T;W^{-2,q}(\Omega))}\leq \Cr{Crt}(T)\label{estT}
\end{align}
for $\ve\in(0,\Cr{e1})$.
\end{Lemma}

\begin{proof}
 Let $T \in (0,\infty)$ 
and consider a relatively open set $$B \subset\subset { \{\D>0\}}.$$  Then, there exists a sufficiently small number $\alpha>0$
such that
$$\hat{B}:=O_{\alpha}(B)  \subset\subset {\R^d \setminus \{D\ngtr 0\}}.$$
Due to the uniform continuity of $\D$ and our choice of $B$,
\begin{align} \label{estdelta}
 \delta:=\underset{x\in \hat B}{\inf}\,\underset{z\in S_1(0)}{\min}z^T\D(x)z>0.
\end{align}
In what follows we leave out the dependence on $\delta$ and $\hat{B}$ in the constants as they are determined by $B$.
Let $\varphi \in C_{0}^{\infty}(\hat B)$ be s.t.
\begin{align}\label{deffi}
	\varphi  \begin{cases}
		= 1 & \text{ on } B,\\
		\in [0,1] & \text{ on } \hat{B}\setminus B,\\
		= 0 & \text{ on } \R^d\setminus \hat{B}.
	\end{cases}
\end{align}
Let $\ve\in(0,\Cr{e2}(B))$, where $\Cr{e2}(B)$ is from \cref{lemapproxd}. We multiply the first equation in \cref{modelapprox} by $\ce\varphi^2$, integrate over $\Omega$ by parts where necessary, using the no-flux boundary condition{,} and obtain
\begin{align}
	\frac{d}{dt} \io \ce^2\varphi^2 \td{x}
	=& - \io \left(\De \nabla \ce +\nabla \cdot \De \ce - \ce \As \ce \right) \cdot \left(\varphi^2\nabla \ce + 2 \ce \varphi \nabla \varphi \right) \td{x} + \mu \io \ce^2\varphi^2 (1-\ce^{r{ -}1}) \td{x}\nonumber\\
	\leq& - \io (\nabla \ce)^T \De \varphi^2\nabla \ce \td{x}  + 2 \left| \io \left(\varphi \nabla \ce\right)^T \De \left(\ce \nabla \varphi\right) \td{x} \right| \nonumber\\
	& + \left|\io \left(\nabla \cdot \De - \As \ce\right) \cdot \left(\ce\varphi^2\nabla \ce + 2 \ce^2 \varphi \nabla \varphi\right) \td{x} \right|  + \mu \io \ce^2\varphi^2 \td{x} - \mu \io \ce^{r+1}\varphi^2 \td{x}. \label{est0}
\end{align}
The matrix $\De$ is symmetric and positive-definite. Hence, it defines a scalar-product on $\R^d$, and with the Cauchy-Schwartz and Young's inequalities and \cref{estDe} we obtain
\begin{align}
	2 \left| \io \left(\varphi \nabla \ce\right)^T \De \left(\ce \nabla \varphi\right) \td{x} \right|
	\leq& 2\sqrt{\io \left(\varphi \nabla \ce\right)^T \De \varphi \nabla \ce \td{x} }\sqrt{\io \left(\ce \nabla \varphi\right)^T \De \ce \nabla \varphi \td{x} } \nonumber\\
	\leq& \frac{1}{4} \io \left(  \nabla \ce\right)^T \De \varphi^2\nabla \ce \td{x} + 4 \|\nabla \varphi\|^2_{L^{\infty}(\R^d)} \Cl{bde} \io {\ce^2} \td{x}, \label{est1}
\end{align}
where $\Cr{bde} := { \Cr{e1}} + \|\D\|_{(L^{\infty}(\Omega))^{d\times d}}$.
Moreover, combining \cref{bounddivepsb}, \cref{lemboundlelr}, and H\"older's and Young's inequalities, we can estimate
\begin{align}
	&\left|\io \left(\nabla \cdot \De - \As \ce\right) \cdot \left(\ce\varphi^2\nabla \ce + 2 \ce^2 \varphi \nabla \varphi\right) \td{x} \right|\nonumber\\
	\leq& \left( \| { |}\nabla\cdot \De { |}\|_{L^{\infty}(\hat{B})} + \|{|\nabla H|}\|_{L^{\infty}(B_{1})} \|\ce\|_{L^1(\Omega)}\right) \io \left|\ce\varphi^2\nabla \ce\right| + 2 \left|\ce^2 \varphi \nabla \varphi\right| \td{x}\nonumber\\
	\leq& \frac{\delta}{2} \io |\varphi\nabla \ce|^2 \td{x} + \Cl{ca2}({ B})\io \ce^2 \td{x}, \label{est2}
\end{align}
Combining \cref{est1,est2,est0} and rearranging the terms leads due to \cref{estdelta} and \cref{deffi} to
\begin{align} \label{estce2v}
	\frac{d}{dt} \io \ce^2\varphi^2 \td{x} + \frac{\delta}{4} \io |\nabla \ce |^2 \varphi^2\td{x} + \mu \io \ce^{r+1}\varphi^2 \td{x} \leq& \Cl{ca3}({ B})\io {\ce^2}\td{x}.
\end{align} 
We integrate \cref{estce2v} over $(0,T)$ and obtain {using $r\geq2$, \cref{lemboundlelr}, \cref{deffi}, the uniform boundedness of the initial values, and H\"older's inequality} that 
\begin{align}\label{boundsalle}
	\|\ce\varphi\|_{L^{\infty}(0,T;L^2(\Omega))}^2 + \frac{\delta}{4} \|\varphi\nabla \ce\|_{L^2(\Omega\times(0,T))}^2 + { \mu} \|\ce\|_{L^{r+1}(B\times(0,T))}^{r+1} \leq \C(B,T),
\end{align}
which yields \cref{estLA_,est3r}.

 Finally, from
\begin{align*}
	\io \partial_t\ce \psi dx
	=& \io \ce \De: D^2\psi dx + \io \ce (\As \ce ) \cdot \nabla \psi dx  + \mu \io \ce (1-\ce^{r-1})\psi dx,
\end{align*} 	
which holds, in particular, for $\psi \in W_0^{2,\frac{q}{q-1}}(\Omega) {\subset C_0^1(\overline{\Omega})}$, 
we obtain using \cref{BndCr32,Assq,Assdr} and the Sobolev embedding theorem that
\begin{align}
 \|\partial_t \ce\|_{W^{-2,q}(\Omega)} \leq \C&\left(d^2 \|\De\|_{(L^{\infty}(\Omega))^{d\times d}}\|\ce\|_{L^r(\Omega)} + \|\nabla H\|_{(L^{\infty}({B_1}))^{d}} \|\ce\|_{L^1(\Omega)}^{ 2}\right.\nonumber\\
 &\ \left.+ \mu \|\ce\|_{L^1(\Omega)}+\mu \|\ce\|_{L^r(\Omega)}^r\right),
\end{align}
yielding \cref{estT} upon integration over $(0,T)$.
\end{proof}

 With the obtained estimates at hand we can now proceed to establishing convergence.
\begin{Lemma} \label{lemconvl1}
There exist $c\in L^{\infty}(0,\infty;L^1(\Omega))$ and a sequence $(\ve_k)\subset(0,\Cr{e1})$, $\ve_k\to0$, s.t.
\begin{align}
	c_{\ve_k} \underset{k \rightarrow \infty}{\rightarrow} c\qquad &\text{in } {L^1_{loc}(\oa \times [0,\infty))},\label{convL1}\\ &\text{a.e. in }\Omega\times (0,\infty).\label{convae}
\end{align}
\end{Lemma}
\begin{proof}
Let $T>0$ and consider a relatively open set $B \subset\subset { \{\D>0\}}$. Thanks to estimates \cref{estLA_,estT}, the dense embeddings
\begin{align*}
 H^1(B) \subset\subset L^2(B)\subset W^{-2,q}(\Omega),
\end{align*}
where the latter holds due to \cref{Assq}, 
and the Lions-Aubin lemma (see e.g Corollary 4 in \cite{Simon}){,} the family $(\ce)$ is precompact in $L^2(\Omega\times(0,T))$.
Consequently, every sequence $(c_{\ve_j})_{j \in \N}$ has a subsequence that converges in $L^2(\Omega\times(0,T))$, and it can be chosen such that it converges a.e. in $B\times(0,T)$.

 Observe that since ${ \{\D>0\}}$  is relatively open in the compact set $\oa$, there exists a sequence $(B_i)$ of relatively open sets such that $B_i\subset\subset{ \{\D>0\}}$ and ${ \{\D>0\}}=\bigcup_{i=1}^{\infty} B_i$. 
In view of this, we have 
\begin{align*}
 { \{\D>0\}}\times(0,\infty)=
 \bigcup_{i=1}^{\infty} B_i\times(0,i).
\end{align*}
Together with a diagonal argument this description in the form of a countable union allows to conclude from the above that there {exist} some $c\in L^{2}_{loc}({(\{\D > 0\}  \times [0,\infty)})$ and a sequence $(\ve_k)\subset(0,\Cr{e1})$ that converges to zero and is such that 
\begin{align}
	c_{\ve_k} \underset{k \rightarrow \infty}{\rightarrow} c\qquad \text{ a.e. in } { \{\D>0\}}\times(0,\infty).\label{convaeT_}
\end{align}
Since by \cref{bedd} the degeneracy set $\{\D \ngtr 0\}$ has the $d$-dimensional Lebesgue measure zero, \cref{convaeT_} is equivalent to
\begin{align}
	c_{\ve_k} \underset{k \rightarrow \infty}{\rightarrow} c\qquad \text{ a.e. in }\Omega\times(0,\infty).\label{convaeT}
\end{align}
Furthermore, due to {\cref{BndCr32}} and $r>1$, the sequence $(c_{\ve_k})_{k\in\N}$ is uniformly integrable on $\Omega\times(0,T)$ for all $T>0$ by the de la Vall\'ee-Poussin theorem (Theorem 22 in \cite{DellacherieMe}). Now  
$$c_{\ve_k} \underset{k\rightarrow\infty}{\rightarrow} c \text{ in } L^1(\Omega\times(0,T))\qquad \text{for all }T>0$$  
and $c \in L^1_{loc}(\oa \times [0,\infty))$ follows with \cref{convaeT} and Vitali's lemma (Theorem 21 in \cite{DellacherieMe}).
\end{proof}

 In preparation for the proof of existence of a very weak solution to \cref{model} we still need one more lemma that allows us to handle the nonlinear part of the reaction term in \cref{modelapprox}.

\begin{Lemma} \label{lemconvr}
Let $(\ve_k)_{k\in\N}$ be as in \cref{lemconvl1}. Then, for all $T\in(0,\infty)$ it holds that {$c \in L^r_{loc}(\oa \times [0,\infty))$} and
\begin{align}
	\cek^r \underset{k \rightarrow \infty}{{\to}} c^r \qquad\text{in }L^1(\Omega\times (0,T)).\label{Conv4}
\end{align}
\end{Lemma}
\begin{proof}
Let $T\in (0,\infty)$ and consider the sequence $(\cek)_{k\in\N}$ from \cref{lemconvl1}. 
Fatou's lemma together with estimate \cref{BndCr32} from \cref{lemboundlelr} imply that
\begin{align*}
	\int_0^T \io c^r \, \td{x} \td{t} \leq \underset{k\rightarrow\infty}{\lim\inf} \int_0^T \io \cek^r \, \td{x} \td{t} \leq \Cr{coeb},
\end{align*}
and so $c\in L^r(\Omega\times (0,T))$.\\  
 Due to \cref{LemFD} and assumption \cref{bedd},  there exists a family of functions $(\vd)_{\delta \in (0,1)} \subset C_c^{\infty}(\R^d;[0,1])$ satisfying \cref{AssPhi} for $K := \{\D \ngtr 0\}$. We adopt the splitting 
\begin{align}
	\cek^r  =  \cek^r (1-\vd)+ \cek^r \vd\label{split}
\end{align}
and next study the convergence of each of the terms separately.

\begin{proofpart}[Convergence of the first term] 
Set 
$$B_{\delta} := {\{\D >0\}} \setminus \overline{O_{{\delta\sqrt{d}}}(\{\D \ngtr 0\})}.$$
 Obviously, $B_{ \delta}$ is relatively open and satisfies $B_{\delta}\subset\subset {\{\D >0\}}$. Arguing similar to the proof of \cref{lemconvl1}, we conclude with \cref{est3r} in \cref{lembounds} and \cref{convae} in \cref{lemconvl1} and the de la Vall\'ee-Poussin theorem (Theorem 22 in \cite{DellacherieMe}) and Vitali's lemma (Theorem 21 in \cite{DellacherieMe})  that
\begin{align} \label{convad}
	\cek^r  \underset{k \rightarrow \infty}{\rightarrow}  c^r \qquad\text{in } L^1(B_{\delta}\times (0,T)).
\end{align}
Since $\vd = 1$ outside of $B_{\delta}$ due to \cref{vd2}, \cref{convad} yields
\begin{align}
 \cek^r(1-\vd) \underset{k\to\infty}{\to}c^r(1-\vd)\qquad\text{in }L^1(\Omega\times(0,T))\text{ for all }\delta\in(0,1).\label{convInt1}
\end{align}
 Furthermore, the integrability of $c^r$  together with  \cref{ApconvPhi} and the uniform boundedness of $(\vd)$ allow to conclude using the dominated convergence theorem that 
\begin{align}
 c^r(1-\vd)\underset{\delta\to0}{{\to}}c^r\qquad\text{in }L^1(\Omega\times(0,T)).\label{convInt2}
\end{align}
\end{proofpart}
\begin{proofpart}[Convergence of the second term and conclusion]
{Due to \cref{distdod,vd3}, we have 
$$\text{supp}(\varphi_{\delta}) \subset \subset O_{5\delta \sqrt{d}}(\{\D \ngtr 0\}) \subset \subset \Omega$$
for $\delta\in(0,1)$ sufficiently small. For such $\delta$, 
}we multiply \cref{IPDEe} by $\vd$ and integrate over $\Omega$, once/twice by parts where necessary, so as to shift all spatial derivatives to $\vd$. Using Hoelder's inequality, \cref{estDe}, and \cref{lemboundlelr,LemFD}, we estimate as follows:
\begin{align*}
	&\frac{d}{dt} \io \cen \vd \td{x} + \mu \io \cen^r \vd \td{x}\\
	\leq& \mu \io \cen \vd \td{x} + \io |\cen\Dek:D^2\vd| \td{x} + \io |\cen \As \cen \cdot \nabla \vd| \td{x}\\
	\leq& \mu \io \cen \vd \td{x} + d^2 \|\Dek\|_{{(L^{\infty}(\Omega))^{d\times d}}} \|D^2\vd\|_{{(L^{\infty}(\Omega))^{d\times d}}}\int_{\{D^2 \vd \neq 0\}}  \cen \td{x}\\ 
	&+ \|\cen\|_{L^{\infty}(0,\infty;L^1(\Omega))} \|\nabla H\|_{{(L^{\infty}(B_1))^d}} \|\nabla \vd\|_{{(L^{\infty}(\Omega))^{d}}} \int_{\{\nabla \vd \neq 0\}}  \cen \td{x}\\
	\leq& \mu \io \cen \vd  \td{x} + \left(\|D^2\vd\|_{{(L^{\infty}(\Omega))^{d\times d}}} + \|\nabla \vd\|_{{(L^{\infty}(\Omega))^{d}}} \right)\C\int_{\text{supp}\vd }  \cen \td{x}\\
	\leq&  \mu \io \cen\vd  \td{x} + \delta^{-2}|\text{supp}(\vd)|^{1-\frac{1}{r}}\Cl{cxi1}\|\cen\|_{L^r(\Omega)}.
\end{align*}
We conclude from Gronwall's and H\"older's inequalities and    \cref{lemboundlelr} that 
\begin{align}
	&\io \cen(\cdot,{ T}) \vd \td{x} + \mu \int_0^{{ T}} \io \cen^r \vd  \td{x}
	\td{t}\nonumber\\
	\leq& e^{\mu T}\left(\io c_{0\ve_k}\vd \td{x} + {{\delta^{-2}|\text{supp}(\vd)|^{1-\frac{1}{r}}}}\Cr{cxi1}\int_0^T\|\cen\|_{L^r(\Omega)}\td{t} \right) \nonumber\\
	\leq& e^{\mu T}\left(\io c_{0\ve_k}\vd \td{x} + {{\delta^{-2}|\text{supp}(\vd)|^{1-\frac{1}{r}}}}T^{\frac{r-1}{r}} \Cr{cxi1}\|\cen\|_{L^r(\Omega\times(0,T))} \right) \nonumber\\
	\leq& e^{\mu T}\left(\io c_{0\ve_k}\vd \td{x} + {{\delta^{-2}|\text{supp}(\vd)|^{1-\frac{1}{r}}}}\Cl{cxi1_}(T)\right){\quad\text{for }t\in(0,T)}. \label{estcr}
\end{align}
Combining \cref{estcr,convc0}, we find that
\begin{align}
  \underset{k\to\infty}{\lim\sup}\int_0^t \io \cen^r \vd  \td{x}\td{t}
  \leq  \mu^{-1}e^{\mu T}\left(\io c_0\vd \td{x} + {{\delta^{-2}|\text{supp}(\vd)|^{1-\frac{1}{r}}}}\Cr{cxi1_}(T)\right)\quad\text{for }t\in(0,T),\ \delta\in(0,1).\label{convInt3}
\end{align}
Due to the integrability of $c_0$, \cref{ApconvPhi}, the uniform boundedness of $(\vd)$, and the dominated convergence theorem we have 
\begin{align}\label{convcod}
\lim_{\delta\to0}\io c_0 \vd \td{x}= 0.
\end{align}
Together, \cref{convcod}, \cref{keylim}, and \cref{convInt3} yield
\begin{align}
 \cen^r \vd \underset{k\to\infty}{\to}\underset{\delta\to0}{\to}0\qquad\text{in }L^1(\Omega\times(0,T)).\label{convrest}
\end{align}
Finally, combining \cref{convrest,convInt2,convInt1,split}, we arrive at \cref{Conv4}.
\end{proofpart}
\end{proof}

\begin{Remark}
The assumptions $d \geq 3$ and $r > \frac{d}{d-2}$ from \cref{Assdr} are only required in the proof of  \cref{lemconvr}. Together with \cref{bedd}, they ensure the existence of the $\vd$s due to \cref{LemFD}.
\end{Remark}

 Finally, we can prove our main result on the existence of a very weak solution to the original IBVP \cref{model}.
\begin{proof}[Proof of \cref{veryweaksol}]
Consider the sequence $(\cek)_{k\in\N}$ from \cref{lemconvl1,lemconvr}. Let $\eta \in C_{c}^{2,1}(\oa \times [0,\infty))$ with $\nabla\eta \cdot (\D \nu) = 0$ on $\partial \Omega \times (0,\infty)$. Then, there is $T\in (0,\infty)$ s.t. $\eta \equiv 0$ for $t \geq T$. We multiply \cref{IPDEe} by $\eta$ and integrate over $\Omega\times (0,\infty)$, once or twice by parts where necessary, using the boundary condition on $\eta$ as well as \cref{bce}, and
for all $k\in\N$ conclude that
\begin{align}
	&- \int_0^{\infty} \io \cek  \partial_t \eta \td{x} \td{t} - \io c_{0\ve_k} \eta(\cdot,0)\td{x}\nonumber\\ 
	=&  \int_0^{\infty} \io {\cek \D_{\ve_k} :} D^2\eta \td{x} \td{t} - \int_0^{\infty} \int_{\partial \Omega} \cek \nabla \eta \cdot \left(\D_{\ve_k} \nu\right) \td{\sigma}(x)\td{t}\nonumber \\
	&+ \int_0^{\infty} \io \cek(\As \cek) \cdot \nabla \eta \td{x} \td{t} + \mu \int_0^{\infty} \io \cek (1-\cek^{r-1}) \eta \td{x} \td{t}.\label{vweake}
\end{align}
We first address convergence of $(\cen)$ on $\partial\Omega\times(0,T)$. Observe that since $O_{a/2}(\partial\Omega)\cap\Omega$ is open and precompact in ${ \{\D>0\}}$, we can make use of the uniform boundedness of $(\cek)_{k\in\N}$ in $L^2(0,T;H^1(O_{a/2}(\partial\Omega)\cap\Omega))$ due to \cref{estLA_} and convergence \cref{convL1}, yielding 
\begin{align}
 \cek\underset{k\to\infty}{\rightharpoonup}c\qquad \text{in }L^2(0,T;H^1(O_{a/2}(\partial\Omega)\cap\Omega))\text{ for all }T>0.\label{convweak}
\end{align}
Using the continuity of the trace operator, we conclude with \cref{convweak} that
\begin{align}
 \cek\underset{k\to\infty}{\rightharpoonup}c\qquad \text{in }L^2(0,T;L^2(\partial\Omega))\text{ for all }T>0.\label{convweaktr}
\end{align}
Now convergences \cref{convc0,detod,convL1,Conv4,convweaktr} and  continuity of $\As:L^1(\Omega)\to (L^{\infty}(\Omega))^{ d}$ together with compensated compactness allow to pass to the limit in each term in \cref{vweake}, yielding \cref{WeakF}.
\end{proof}

\section{Smooth very weak solutions are classical} \label{appendixvws}
In this final Section we provide a justification for the very weak formulation \cref{WeakF}.
We show that as in the case of Neumann boundary conditions for elliptic equations (see e.g. Theorem 2.2.2.5 in \cite{Grisvard}) it holds for smooth $\D$ that any sufficiently smooth very weak solution (in terms of  \cref{defvws}) is also a classical solution to \cref{model}. 
\begin{Theorem}
In addition to \cref{assump}, let 
\begin{align}
	&\D\in C^{2}(\oa;\R^{d\times d}),\nonumber\\
	&c\in C^{2,1}(\oa \times (0,\infty))\cap C(\oa \times [0,\infty)).
\end{align}
 Then, if $c$ is a solution to \cref{model} in the sense of \cref{defvws},
then it solves this IBVP in the classical sense.
\end{Theorem}
\begin{proof}
Let 
\begin{align}
&\eta \in C_c^{2,1}(\oa \times [0,\infty)) \qquad\text{s.t. } \nabla\eta \cdot \left(\D \nu\right) = 0 \text{ on } \partial \Omega \times (0,\infty). \label{Asseta}                                                                                                                                           \end{align}
Then, there is $T\in (0,\infty)$ s.t. $\eta \equiv 0$ for $t \geq T$.
Integrating by parts on both sides of \cref{WeakF}, twice where necessary, using the information about  $c$ and $\eta$ on $\partial\Omega\times(0,T)$ that we have, yields
\begin{align}
 \int_0^{\infty} \io \partial_t c \eta \td{x} \td{t} 
	=&- \int_0^{\infty} \io  \nabla\cdot(c\D)\cdot \nabla\eta\td{x} \td{t} - \int_0^{\infty} \io \nabla\cdot(c(\As c))  \eta \td{x} \td{t}+\mu \int_0^{\infty} \io c (1-c^{r-1}) \eta \td{x} \td{t} \nonumber\\
	&+\int_0^{\infty}\int_{\partial\Omega}c(\As c)\cdot\nu\eta\td{\sigma}(x) \td{t} \nonumber\\
	=&\int_0^{\infty} \io  \left(\nabla\nabla:(c\D)-\nabla\cdot(c(\As c))+ \mu c (1-c^{r-1}) \right)\eta\td{x} \td{t}\nonumber\\
	&-\int_0^{\infty}\int_{\partial\Omega}\left(\nabla\cdot(\D c)-c(\As c)\right)\cdot\nu\eta\td{\sigma}(x) \td{t}.\label{eqc}
\end{align}
For the subset of $\eta \in C_c^{2,1}(\Omega\times (0,\infty))$ the boundary integral in  \cref{eqc} vanishes. Thus, the variational lemma applies and yields that $c$ satisfies 
\begin{align*}
	\partial_t c = \nabla \nabla : (\D(x)c) - \nabla \cdot (c \As c ) + \mu c (1-c^{r-1})
\end{align*}
pointwise in $\Omega\times (0,\infty)$. Now we can conclude from \cref{eqc} that for all $\eta$ satisfying \cref{Asseta} it holds that
\begin{align}\label{intrand}
	\int_0^{\infty}\int_{\partial\Omega}\left(\nabla\cdot(\D c)-c(\As c)\right)\cdot\nu\eta\td{\sigma}(x) \td{t}=0.
\end{align}
{From now on,} we consider $\eta$ of the form $\eta(x,t)=\eta_1(x)\eta_2(t)$, where $\eta_{{2}} \in C_c^{1}([0,\infty))$ and $\eta_{{1}} \in C^{2}(\oa)$ satisfies $\nabla\eta_{{1}} \cdot \left(\D \nu\right) = 0$ on $\partial \Omega$. Applying the variational lemma w.r.t. to time integration in  \cref{intrand}  yields
\begin{align}\label{intrandx}
	\int_{\partial\Omega}\left(\nabla\cdot(\D c)-c(\As c)\right)\cdot\nu\eta_{{1}}\td{\sigma}(x) =0\qquad\text{ for all }t\in(0,T).
\end{align}
The boundary condition \cref{bc} now follows with \cref{intrandx} and \cref{lemdense} below.
\end{proof}
\begin{Lemma} \label{lemdense}
	Let $\D\in C^{2}(\oa;\R^{d\times d})$ {symmetric}, $\D\geq 0$ with $\{\D\ngtr 0\}\cap \partial\Omega=\emptyset$. Then, the set
	\begin{align*}
		\{\eta \in {C^2(\oa)}: \nabla \eta \cdot (\D\nu)= 0 \text{ on }\partial\Omega\}
	\end{align*}
	is dense in $H^1(\Omega)$.
\end{Lemma}
\begin{proof}
	Since $\D$ is continuous and positive definite in $\partial\Omega$, it is positive definite in some open neighbourhood of $\partial\Omega$. Choose some { symmetric} $\mathbb{B} \in C_c^{\infty}(\Omega;\R^{d\times d})$, $\mathbb{B}\geq 0$, and positive definite in an open neighbourhood of $\{\D\ngtr 0\}$. Then, $$\tilde{\D} := \D + \mathbb{B}$$ satisfies
	$$\D = \tilde{\D}$$ in some open neighbourhood of $\partial \Omega$, and there is some $\delta>0$ s.t. $$y^T\tilde{\D}(x)y \geq \delta |y|^2\qquad\text{ for all }x\in \oa,\ y \in \R^d.$$ 
	In $H^1(\Omega)$, consider the scalar product
	\begin{align}
		\langle f,g \rangle := { \lambda} \io fg \td x + \io(\nabla f)^T \tilde{\D} \nabla g \td x\label{scpr}
	\end{align} 
	for some $\lambda>0$. Since $\tilde D$ is smooth and positive definite in $\overline{\Omega}$, the norm induced by \cref{scpr} is equivalent to the standard norm on $H^1(\Omega)$. 
	Set
	\begin{align*}
		E :=& \{\eta \in C^2(\oa):\quad \nabla \eta \cdot (\D\nu)= 0 \text{ on }\partial\Omega\}\\
		=&\{\eta \in C^2(\oa):\quad\nabla \eta \cdot (\tilde\D\nu)= 0 \text{ on }\partial\Omega\}.
	\end{align*}
	We thus need to verify that the orthogonal compliment of $E$ w.r.t. to the above scalar product in $H^1(\Omega)$ only contains the zero vector.
	Assume the contrary, i.e. that
	\begin{align*}
		\overline{E}^{\perp} = \{\xi \in H^1(\Omega):\quad \langle \xi, \eta \rangle = 0 \text{ for all } \eta \in \overline{E}\}\neq \{0\} .
	\end{align*}
	Let $\xi\in \overline{E}^{\perp}$.  Then, as $\partial \Omega$ is sufficiently smooth, there is a sequence $(\xi_n)_{n\in\N} \subset C^{\infty}(\oa)$ s.t. $\xi_n \underset{n\to\infty}{\rightarrow}\xi$ in $H^1(\Omega)$. Consider the sequence of elliptic BVPs
	\begin{subequations}\label{elBVP}
	\begin{alignat}{3} 
    -\nabla\cdot (\tilde{\D}\nabla u_{ n}) + \lambda u_{ n} =& \xi_{ n} &&\qquad\text{in } \Omega,\\
			\nabla u_{ n} \cdot (\tilde\D\nu) =& 0 &&\qquad\text{on }\partial \Omega.
	\end{alignat}
	\end{subequations}
	Theorems 2.3.3.6, 2.4.2.7 and 2.5.1.1 in \cite{Grisvard} imply that for sufficiently large $\lambda>0$ there exists a solution $u_n \in C^2(\oa)$ to \cref{elBVP} for all $n\in\N$, and $(u_n)_{n\in\N}$ is uniformly bounded in $H^2(\Omega)$. Consequently, due to the continuity of the trace operator and the compact embedding of $H^2$ in $H^1$ there exist a sequence $(n_l)$ and some $u \in H^2(\Omega)$ s.t.
	\begin{align*}
		u_{n_l} \underset{l\to \infty}{\rightharpoonup} u &\qquad\text{in } H^2(\Omega),\\
		u_{n_l} \underset{l\to \infty}{\rightarrow} u &\qquad\text{in } H^1(\Omega),\\
		u_{n_l} \underset{l\to \infty}{\rightharpoonup} u &\qquad\text{in } H^1(\partial\Omega).
	\end{align*}	
	It follows that  $u\in  \overline{E}$ and is a strong $L^2$ solution to the IBVP
	\begin{subequations} \label{elIBVPGW}
	\begin{alignat}{3} 
    -\nabla\cdot (\tilde{\D}\nabla u) + \lambda u =& \xi &&\qquad\text{in } \Omega,\label{BP}\\
			\nabla u \cdot (\tilde\D\nu) =& 0 &&\qquad\text{on }\partial \Omega.\label{BPbc}
	\end{alignat}
	\end{subequations}
 	Multiplying  \cref{BP} by $\xi$ and integrating by parts using \cref{BPbc} then yields
	\begin{align*}
		0 = \langle u, \xi \rangle = { \lambda} \io u\xi \td x + \io(\nabla u)^T \tilde{\D} \nabla \xi \td x = \io \xi^2 \td x. 
	\end{align*}
	This shows that $\xi \equiv 0$, contradicting the above assumption. Therefore, $\overline{E}=H^1(\Omega)$, as required.
\end{proof}
\section*{Acknowledgement}
\addcontentsline{toc}{section}{Acknowledgement}
\begin{itemize}
\item 
AZ was supported by the Engineering and Physical Sciences Research Council [grant number
EP/T03131X/1].
\item For the purpose of open access, the authors have applied a Creative Commons Attribution (CC BY) licence to any Author Accepted Manuscript version arising.
\item No new data were generated or analysed during this study.                                                       \end{itemize}
\phantomsection
\printbibliography

\crefalias{section}{appendix}\begin{appendices}
\section{A lemma about sets of ``sufficiently small'' dimension}\label{appB}
We recall one of the (alternative) ways of defining fractional  dimension, see (3.5)  and the subsequent discussion on p. 42 in  \cite{Falconer}.
\begin{Definition}[Upper box  dimension]\label{DefUB}
 Let $K\subset\R^d$ be compact. For every $\delta>0$ we denote 
 \begin{align}
  Z_{\delta}(K):=\left\{ b\in\delta\Z^d:\  |x-b|_{\infty}\leq\frac{\delta}{2}\text{ for some }x\in K\right\}.\nonumber
 \end{align}
 The upper box dimension is the non-negative number 
 \begin{align*}
  \overline{\dim}_{{\cal F}}(K):=\underset{\delta\to0}{\lim\sup}\frac{\log_2|Z_{\delta}(K)|}{\log_2\delta^{-1}}.
 \end{align*}

\end{Definition}

\begin{Lemma}\label{LemFD}
Let 
\begin{align}
 &3\leq d\in\N,\\
 &r>\frac{d}{d-2},
\end{align}
and $K\subset\R^d$ be a compact set such that 
\begin{align}
  \overline{\dim}_{{\cal F}}(K)<d-\frac{2r}{r-1}.\label{Assdim}
\end{align}
Then, there exists a family $(\vd)_{\delta\in(0,1)}$ of functions such that for all $\delta>0$
\begin{subequations}\label{AssPhi}
\begin{align}
&\vd \in C_{c}^{\infty}(\R^d;[0,1]), \label{vd1}\\
 &\vd=1\qquad\text{in }\overline{O_{{\delta\sqrt{d}}}(K)},\label{vd2}\\
 &\supp (\vd)\subset \overline{O_{5\delta\sqrt{d}}(K)},\label{vd3}\\
 &\|\nabla\vd\|_{(L^{\infty}(\R^d))^d}\leq \delta^{-1}\Cl{D_1},\label{der1}\\
 &\|D^2\vd\|_{(L^{\infty}(\R^d))^{d\times d}}\leq \delta^{-2}\Cr{D_1},\label{der2}\\
 &\vd\underset{\delta\to 0}{\to}0\qquad\text{a.e. in }\R^d,\label{ApconvPhi}\\
 &\underset{\delta\to0}{\lim}\,\delta^{-\frac{2r}{r-1}}|\supp(\varphi_{ \delta})|=0.\label{keylim}
\end{align}
\end{subequations}
\end{Lemma}
\begin{proof}
 Let $\eta\in C_0^{\infty}([0,\infty);[0,1])$ be such that 
 \begin{align*}
  \eta=\begin{cases}
        1&\text{ in }[0,1],\\
        0&\text{ in }[2,\infty).
       \end{cases}
 \end{align*}
Set
\begin{align}
 \vd(x):=\frac{\sum\left\{\eta\left(\frac{|x-z|}{\delta\sqrt{d}}\right):\ z\in \delta\Z^d\cap O_{3\delta\sqrt{d}}(K)\right\}}{\sum\left\{\eta\left(\frac{|x-z|}{\delta\sqrt{d}}\right):\ z\in \delta\Z^d\right\}}\qquad \text{for }x\in\R^d,\ \delta\in(0,1).\label{defvar}
\end{align}
We need to check that $\vd$ satisfies the required properties.
\begin{enumerate}
\item\label{item1}  Since 
\begin{align}
 \max_{x\in\R^d}\left|B(x;3\delta\sqrt{d})\cap \delta\Z^d\right|=&\max_{y\in\R^d}\left|B(y;3\sqrt{d})\cap \Z^d\right|\nonumber\\
 =: &\Cl{D2}<\infty\nonumber
\end{align}
for some $\Cr{D2}>0$, 
and $\eta= 0$ in $[2,\infty)$, the sums in \cref{defvar} contain at most $\Cr{D2}$ non-zero summands.
 \item\label{item2} Since $\R^d= O_{\delta\sqrt{d}}(\delta\Z^d)$ and $\eta= 1$ in $[0,1]$, the denominator in \cref{defvar} is never zero, and
 \begin{align}
  \sum\left\{\eta\left(\frac{|x-z|}{\delta\sqrt{d}}\right):\ z\in \delta\Z^d\right\}\geq1.\label{denom}
 \end{align}

 \item By \cref{item1,item2}  and the assumptions on $\eta$, function $\vd$ is well-defined and belongs to $C_0^{\infty}(\R^d;[0,1])$. 
 \item Since $\eta=0$ in $[2,\infty)$, the numerator and the denominator coincide in $O_{\delta\sqrt{d}}(K)$, hence $\vd=1$ there.
 \item Since $\eta=0$ in $[2,\infty)$,
 \begin{align}
\supp(\vd)\subset  &\overline{O_{2\delta\sqrt{d}}\left(\delta\Z^d\cap O_{3\delta\sqrt{d}}(K)\right)}\nonumber\\
\subset& \overline{O_{5\delta\sqrt{d}}(K)}\nonumber \\
\subset& \overline{O_{6\delta\sqrt{d}}(Z_{\delta}(K))}.  \label{suppvd}                                  \end{align}
Combining \cref{suppvd,Assdim}, we obtain 
\begin{align}
 \log_2\left(\delta^{-\frac{2r}{r-1}}|\supp(\vd)|\right)\leq &\log_2\left(\delta^{-\frac{2r}{r-1}}|Z_{\delta}(K)|(6\delta\sqrt{d})^d|B_1|\right)\nonumber\\
 \leq&\log_2\left(\Cl{D1} \delta^{d-\frac{2r}{r-1}}|Z_{\delta}(K)|\right)\nonumber\\
 =&\log_2\Cr{D1}+\log_2\delta^{-1}\left(\frac{\log_2|Z_{\delta}(K)|}{\log_2\delta^{-1}}-\left(d-\frac{2r}{r-1}\right)\right)\nonumber\\
 \underset{\delta\to0}{\to}&-\infty,
\end{align}
so \cref{keylim} holds.
\item We compute
\begin{align}
 \nabla\vd(x)=&(\delta\sqrt{d})^{-1}\sign(x-z)\left(\frac{\sum\left\{\eta'\left(\frac{|x-z|}{\delta\sqrt{d}}\right):\ z\in \delta\Z^d\cap O_{3\delta\sqrt{d}}(K)\right\}}{\sum\left\{\eta\left(\frac{|x-z|}{\delta\sqrt{d}}\right):\ z\in \delta\Z^d\right\}}\right.\nonumber\\
 &-\left.\frac{\sum\left\{\eta\left(\frac{|x-z|}{\delta\sqrt{d}}\right):\ z\in \delta\Z^d\cap O_{3\delta\sqrt{d}}(K)\right\}\left(\sum\left\{\eta'\left(\frac{|x-z|}{\delta\sqrt{d}}\right):\ z\in \delta\Z^d\right\}\right)}{\left(\sum\left\{\eta\left(\frac{|x-z|}{\delta\sqrt{d}}\right):\ z\in \delta\Z^d\right\}\right)^2}\right).\label{der_1}
\end{align}
Combining \cref{item1,item2,der_1} and the assumptions on $\eta$, we obtain \cref{der1}. Differentiating again and using the same argument yields \cref{der2}.
\item The convergence \cref{ApconvPhi} is a direct consequence of \cref{vd2,vd3} and $\overline{\dim}_{{\cal F}}(K)<d$. 
\end{enumerate}

\end{proof}

\end{appendices}

\end{document}